\documentclass[11pt, oneside,reqno]{amsart}   	
\usepackage{geometry}                		
\usepackage{graphicx}				
\usepackage{amssymb}
\usepackage{breqn}
\usepackage{mathtools}

\title{Endoscopic Transfer for unitary Lie Algebras}
\author{Jingwei Xiao}
\address{MIT, Massachusetts, USA, 02139.} 
\email{jwxiao@mit.edu}

\newtheorem{theorem}{Theorem}[section]
\newtheorem{definition}[theorem]{Definition}

\newtheorem{lemma}[theorem]{Lemma}
\newtheorem{proposition}[theorem]{Proposition}
\newtheorem{conjecture}[theorem]{Conjecture}
\newtheorem*{remark}{Remark}
\newtheorem*{acknowledgement}{Acknowledgement}

\begin{document}
\bibliographystyle{plain}
\maketitle

\begin{abstract}
We give another proof of the existence of the endoscopic transfer for unitary Lie algebras and its compatibility with Fourier transforms. By the work of Kazhdan and Vashavsky in \cite{Kazhdan}, this implies the corresponding endoscopic fundamental lemma (theorem of Laumon--Ng\^o).  We study the compatibility between Fourier transforms and transfers and we prove that the compatibility in the Jacquet-Rallis setting implies the compatibility in the endoscopic setting for unitary groups.
\end{abstract}

\tableofcontents

\section{Introduction}

In \cite{Wa1} and \cite{Wa2}, Waldspurger states a conjecture that the Langlands-Shelstad endoscopic transfer on Lie algebras always exists and is preserved by Fourier transforms:

\begin{conjecture} \label{MC}
Let $F$ be a non-archimedean local field of characteristic zero, $G$ a reductive group over $F$ and $H$ its endoscopic group. Let $\mathfrak{g}$ and $\mathfrak{h}$ be the Lie algebras of $G$ and $H$.

$\mathbf{1}$. For any $f \in C_c^\infty(\mathfrak{g})$, there exists a transfer $f^H \in C_c^\infty(\mathfrak{h})$.

$\mathbf{2}$. Take compatible Fourier transforms $\mathcal{F}_G$ and $\mathcal{F}_H$ on  $\mathfrak{g}$ and $\mathfrak{h}$. There exists a constant $c$ such that whenever $f^H$ is a transfer of $f $, then
$\mathcal{F}_H(f^H)$ is a transfer of $c\mathcal{F}_G(f)$. 

\end{conjecture}

Waldspurger shows in \cite{Wa2} that conjecture \ref{MC} follows from the fundamental lemma. In this paper, we consider the case when $G=\textup{U}(W)$ and $H=\textup{U}(W_a) \times U(W_b)$. We give another proof of conjecture \ref{MC} in this case without using the fundamental lemma. Here $W, W_a, W_b$ are Hermitian spaces for a quadratic extension $E$ over $F$. 

In \cite{Kazhdan}, Kazhdan and Varshavsky have shown that the fundamental lemma is also implied by conjecture \ref{MC}. So our paper can also be regraded as another proof of the fundamental lemma. 

\bigbreak

We explain the idea of our proof. The key theorem is an identity between nilpotent orbit integrals in the Jacquet-Rallis setting. Let $E/F$ be a quadratic extension of non-archimedean local fields of characteristic zero, $V$ a vector space over $F$ of dimension $n$, and $W_0$, $W_1$ the two isomorphism classes of Hermitian spaces of dimension $n$ for $E/F$.

The Jacquet-Rallis transfer is transfer between $f \in C_c^\infty(\mathfrak{gl}(V)\times V \times V^*)$ and a pair $(f_0, f_1)$ where $f_i \in C_c^\infty(\mathfrak{u}(W_i) \times W_i)$ such that their orbit integral in the \textit{regular semisimple} matching orbits are the same. Here, the orbit is taken for the diagonal action of GL($V$) on $\mathfrak{gl}(V)\times V \times V^*$ and the diagonal action of U$(W_i)$ on $\mathfrak{u}(W_i) \times W_i$. See section \ref{JR review} for more details.

Given matching functions, we prove the following (theorem \ref{Nilpotent Identity}):

\begin{theorem} \label{XY}
 Choose matching orbits \textup{(}$\gamma,v, v^*$\textup{)} and \textup{(}$\delta,w$\textup{)}. We have the following identity between nilpotent orbit integrals
\begin{equation}
\omega(\gamma,v, v^*) \textup{Orb}(f,(\gamma,v_{\Lambda},v^*_{S_1 \backslash \Lambda}))= \sum_x \langle \Lambda,x \rangle \int_{\textup{U}(W_x)/T_{\delta_x}}f_W(g\delta_x g^{-1},0)d\overline{g}.
\end{equation}
\end{theorem}

We shall explain in a vague way what these mean. On the right side, the summation of $x$ is a summation over $H^1(F,T_\gamma)$. For each $x$, we defines an element $\delta_x$ that is stable conjugate to $\delta$. When $x$ varies, $\delta_x$ gives all the conjugacy classes in the stable conjugacy class of $\delta$. For each $\Lambda$, we define a character $\langle \Lambda, \cdot \rangle$ on $H^1(F,T_\gamma)$. For each character $\kappa$ of $H^1(F,T_\gamma)$, we can find a unique $\Lambda$ such that $\langle \Lambda, x \rangle=\kappa(x)$. Now on the left side, the choice of $\Lambda$ gives a $\gamma$-nilpotent orbit, and we have an orbit integral $\textup{Orb}(f, \cdot )$ along this orbit.
Therefore, the equation reduces any $\kappa$-orbits integrals to a single nilpotent orbit integrals on $\mathfrak{gl}(V)\times V \times V^*$. Using this equation for dimension $n$, $a$ and $b$, we reduce endoscopic identities to identities involving nilpotent orbit integrals on the general linear side, which can be proved by parabolic induction.

Finally, to prove the above equation, the idea is to choose regular semisimple orbits close to the nilpotent orbits. For each choice, we have an identity of orbit integrals of matching functions. Then we take limit and subtract the equation from the limit. The fact that the nilpotent orbit integrals appear in this limit is an consequence of the germ expansion principle.

\bigskip

Since we are proving something that is already known, it is important to explain what results we need to use. We need two ingredients.

The first ingredient is the existence of transfer and its compatibility with Fourier transforms in the Jacquet-Rallis transfer (theorem \ref{JR exist}, theorem \ref{JR commute}). This is proved in \cite{Wei1} by a purely local method without using the fundamental lemma.

The second ingredient is the existence of transfer and its compatibility with Fourier transform in the Jacquet-Langlands transfer between inner forms of unitary Lie algebras (theorem \ref{JL exist}, theorem \ref{JL commute}). This is a special case of conjecture \ref{MC} when $G$ and $H$ are inner forms of unitary groups. This case is proved in \cite{Wa2} without any assumption because fundamental lemmas between inner forms are trivial. Except for the Jacquet-Langlands transfer, our proof is purely local. In particular, we do not reduce to positive characteristic cases. The same method also shows the fundamental lemma in the Jacquet-Rallis case implies the fundamental lemma in the endoscopic case.  Finally, it is interesting to see whether the method of Kazhdan and Varshavsky can be applied directly to the Jacquet-Rallis setting, thus yield another proof of the Jacquet-Rallis fundamental lemma. The author hopes to return to this topic in the future.

\bigbreak

Here is the structure of the paper. In section \ref{JR review}, we review basic facts about the Jacquet-Rallis transfer. In section \ref{LS review}, we review basic facts about the endoscopic transfer and work it out in details in the case of unitary Lie algebras. Then, we prove theorem \ref{XY} in section \ref{Nil}. Finally, we prove our main theorem in section \ref{final}.

Throughout this paper, we work over a non-archimedean local field $F$ of characteristic zero.  And $E$ will be a quadratic extension over $F$.

\begin{acknowledgement}
\textup{I would like to express my great gratitude to my advisor Wei Zhang, for suggesting this problem to me and providing marvelous insights. I also want to thank Professor Herv\'e Jacquet for helpful discussions when preparing this paper.}

 \end{acknowledgement}

\section{Review of the Jacquet-Rallis transfer} \label{JR review}

In this section, we review the Jacquet-Rallis transfer. The main reference is \cite{Rallis1}.

\subsection{General linear case}

Let $E/F$ be a quadratic extension of non-archimedean local fields and we denote by $\chi$ the corresponding quadratic character of $F^\times$.

Let $V$ be a vector space over $F$ of dimension $n$. We denote by $\langle \ ,  \ \rangle$ the pairing between $V$ and $V^*$. Let Ad($g$) be the adjoint action of $g\in$ GL$(V)$ on $\mathfrak{gl}(V)$. For  $g\in$ GL$(V)$ and $v^* \in V^*$, define $v^*g$ by the formula
$$\langle v^*g,v\rangle=\langle v^*,gv\rangle.$$

Consider the following action of GL($V$) on $\mathfrak{gl}(V)\times V\times V^*$:
$$g.(x,v,v^*)=(\textup{Ad}(g)x,gv,v^*g^{-1}).$$

For $(x,v,v^*) \in \mathfrak{gl}(V)\times V\times V^*$ and $i=0,...,n-1$, define the following invariants:

$$a_i=\textup{coefficient of }t^i \textup{ in } \det(\textup{I}t-x),$$
$$b_i=\langle v^*,x^iv\rangle.$$

Let $\{\langle v^*,x^{i+j}v\rangle\}_{i,j}$ denote the matrix with $(i,j)$ entries $\langle v^*,x^{i+j}v\rangle$. Define

$$\Delta(x,v,v^*)=\det(\{\langle v^*,x^{i+j}v\rangle\}_{i,j}).$$

An element $(x,v,v^*)$ is called regular semisimple if its orbit is closed with maximal dimension. It is known that $(x,v,v^*)$ is regular semisimple precisely when $\Delta(x,v,v^*) \neq 0$ (Theorem 6.1 in \cite{Rallis1}). Also, two regular semisimple elements have the same invariants precisely when they are in the same orbit (Prop. 6.2 in \cite{Rallis1}).

Fix a Haar measuer $dg$ on GL$(V)$, given a function $f \in C_c^\infty (\mathfrak{gl}(V)\times V\times V^*)$ and a regular semisimple element $(x,v,v^*)$, we define its orbit integral

$$\textup{Orb}(f,(x,v,v^*))=\int_{\textup{GL}(V)}f(\textup{Ad}(g),gv,v^*g^{-1})\chi(g)dg.$$

Later, we will need the following generized ``nilpotent orbit integrals''. Fix $x \in \mathfrak{gl}(V)$ regular semisimple in the usual sense. Let $F[x]=\prod_{i=1}^m F_i$, where each $F_i$ is a field. Write $\{1,...,m\}=S_1 \amalg S_2$. Here $S_1$ consists of all $i$ such that $F_i \nsupseteq E$ and $S_2$ consists of all $i$ such that $F_i \supseteq E$. Let $V_i=F_iV$ and $V_i^*=V^*F_i$, so we have the decomposition $V=\bigoplus V_i$ and $V^*=\bigoplus V_i^*$. We write vectors according to this decomposation. 

Consider the orbit of $(x,v,v^*)$, where $v=(v_1,...,v_m)$ and $v^*=(v_1^*,...,v_m^*)$, with the following property:
\bigbreak

\begin{enumerate}

\item For $i \in S_1$, exactly one of the $v_i$ and $v^*_i$ is $0$. 
\item When $v_i \neq 0$, we require that $\{v_i,xv_i,x^2v_i,...,x^{\dim V_i-1}v_i\}$ form a  basis of $V_i$. Similar condition applies when $v^*_i \neq 0$.
\item For each $i \in S_2, v_i=v_i^*=0$.

\end{enumerate}
\bigbreak
These $(x,v,v^*)$ are not regular semisimple. We define its orbit integral in the following way. In the special case when $S_2=\emptyset$, we define $\textup{Orb}(f,(x,v,v^*))$ as 
$$\textup{Orb}(f,(x,v,v^*))=\int_{\textup{GL(V)}}f(\textup{Ad}(g)x,gv,v^*g^{-1})\chi(g)|g|^sdg \bigg \rvert_{s=0}$$

In the general case, let $T$ denote the centralizer of $x$ in GL$(V)$, so $T=F[x]^\times=\prod F_i^\times$. Set $T_1=\prod_{i \in S_1}F_i$. We write $t=(t_1,...t_m)$ for this decomposation. Define

\begin{equation} \label{converge}
\textup{Orb}(f,(x,v,v^*))=\int_{\textup{GL(V)}/T} \bigg(\int_{T_1}f(\textup{Ad}(g)x,gtv,v^*t^{-1}g^{-1})\prod_{i \in S_1}|t_i|^{\pm s}\chi(t)\chi(g)dt \bigg \rvert_{s=0}\bigg) d\overline{g}.
\end{equation}

 We use $+1$ if $v_i \neq 0$ and $-1$ if $v_i^* \neq 0$. This definition makes sense because of the following lemma:

\begin{lemma}
The inner integral over $T_1$ in \textup{(\ref{converge})} converges when \textup{Re($s$)} $> 0$ and  admits meromorphic continuation to the whole complex plane which at holomorphic at $0$. Its value at $0$ is well defined as a locally constant and compactly supported function on $\textup{GL($V$)}/T$.
\end{lemma}
\begin{proof}
In general, given a compactly supported function $f$ on a local field $F$ and $\chi$ a unitary character of $F^\times$, the integral $$\int_{F^\times} f(t)|t|^s\chi(t)dt^\times$$
converge for Re($s$) $> 0$ and admits a meromorphic continuation to the whole complex plane. When $\chi$ is nontrivial, the meromorphic continuation is holomorphic at $0$.

Now to prove the lemma, it suffices to write $f$ in (\ref{converge}) as a product of functions, then the meromorphic continuation and being holomorphic at $0$ are reduced to the above statement.

If we change $g$ to $gt_0$. The vaule of the inner integral in (\ref{converge}) after evaluating at s=0 does not change. So its a well defined function on $\textup{GL$(V)$}/T$. Now since $x$ is regular semisimple, it is easy to check this function is locally constant and compactly supported. 

\end{proof}

\subsection{Unitary case} 

As before, let $E/F$ be a quadratic extension of non-archimedean local fields. There are two isomorphism classes of non-degenerate Hermitian spaces of dimension $n$ for the extension $E/F$. They are distinguished by the determinant of their Hermitian matrices as an element in $F^\times/\textup{Nm}_{E/F}(E^\times)$. Let $W$ be one of the two Hermitian spaces. Let $\langle \ , \ \rangle$ be the Hermitian pairing and U$(W)$ the corresponding unitary group over $F$. Throughout this paper, let $\mathfrak{u}(W)$ be the \textit{twisted} Lie algebra of U$(W_i)$ over $F$. Namely,

$$\mathfrak{u}(W)=\{x \in \textup{End}(W)|\langle xu , v\rangle=\langle u , xv\rangle \textup{ for all }u,v \in W \}.$$

Since we will never use the usual Lie algebras of unitary groups, this notation will not cause any confusion. Let $\sqrt{\tau}$ be a generator of the extension $E/F$, then multiplication by $\sqrt{\tau}$ induces an isomorphism between $\mathfrak{u}(W)$ and the usual Lie algebras. We use the twisted version to avoid this dependence on $\tau$ as our statements are more natural for the twisted Lie algebra. In most cases, statements about Lie algebras are trivially equivalent to the statements about twisted Lie algebras. And we will in general ignore justifications of these equivalences. 

We have an adjoint action of U$(W)$ on $\mathfrak{u}(W)$. Consider the action of U$(W)$ on $\mathfrak{u}(W) \times W$ by

$$g.(x,w)=(\textup{Ad}(g)x,gw).$$

For $(x,w) \in \mathfrak{u}(W) \times W$ and $i=0,...,n-1$, define the following invariants:

$$a_i=\textup{coefficients of }t^i \textup{ in }\det(\textup{I}t- x),$$

$$b_i=\langle w, x^iw\rangle.$$

Take a basis of $W$ and let $H$ be the Hermitian matrix with respect to this basis. Then $x \in \mathfrak{u}(W)$ implies $x^tH=H \overline{x}$, equivalently $H^{-1}x^tH=\overline{x}$. Taking characteristic polynomial of both sides, we deduce that $a_i \in F$. Also $$\overline{b_i}=\overline{\langle w, x^iw\rangle}=\langle  x^iw,w \rangle=\langle w, x^iw\rangle=b_i.$$
So $b_i \in F$ as well.

Let $\{\langle w,x^{i+j}w\rangle\}_{i,j}$ denote the matrix with $(i,j)$ entries $\langle w,x^{i+j}w\rangle$. Define

$$\Delta(x,w)=\det(\{\langle w,x^{i+j}w\rangle\}_{i,j}).$$

As in the general linear case, following the same proof as in \cite{Rallis1}, it can be seen that $(x,w)$ is regular semisimple precisely when $\Delta(x,w)\neq 0$. Also, two regular semisimple elements have the same invariants precisely when they are in the same orbit.

When $(x,w)$ is regular semisimple, then $\{x,xw,x^2w,...,x^{n-1}w\}$ form a basis of $W$ and $\{\langle w,x^{i+j}w\rangle\}_{i,j}$ is the Hermitian matrix with respect to this basis. In particular, $\chi(\Delta(x,w))$ only depends on the isomorphism class of the Hermitian space $W$.

Fix a Haar measure $dg$ on U$(W)$, given a function $f \in C_c^\infty(\mathfrak{u}(W)\times W)$ and a regular semisimple $(x,w)$, we define the orbit integral

$$\textup{Orb}(f,(x,w))=\int_{\textup{U}(W)}f(\textup{Ad}(g)x,gw)dg.$$

\subsection{The transfer} \label{the transfer}

We first define the transfer of orbits. Given a regular semisimple element in $\mathfrak{gl}(V)\times V\times V^*$ and a regular semisimple element in $\mathfrak{u}(W)\times W$, they match if their invariants $a_i$ and $b_i$ are the same for $i=0,...,n-1$. Here $n=\textup{dim}V=\textup{dim}W$. Notice that this definition only involves regular semisimple elements. 

Let $W_0$ and $W_1$ represent the two isomorphism classes of Hermitian spaces of dimension $n$. It is proved in \cite{Wei3} that matching orbits gives a bijection between regular semisimple orbits in $\mathfrak{gl}(V)\times V\times V^*$ and the disjoint union of the regular semisimple orbits in $\mathfrak{u}(W_i)\times W_i$ $(i=0,1)$.

Now we define the transfer of functions. The Jacquet-Rallis transfer of functions is a transfer between $f \in C_c^\infty(\mathfrak{gl}(V)\times V\times V^*)$ and a pair $\{f_0,f_1\}$, where $f_i \in C_c^\infty(\mathfrak{u}(W_i)\times W_i)$ for  $i=0,1$.

We shall define $f$ and $\{f_0,f_1\}$ are transfers if for all the matching orbits $(x,v,v^*)$ and $(y,w)$ where $(y,w) \in \mathfrak{u}(W_i) \times W_i$, the following identity is satisfied:

$$\omega(x,v,v^*)\textup{Orb}(f,(x,v,v^*))=\textup{Orb}(f_i,(y,w)).$$

Here $\omega(x,v,v^*)$ is the transfer factor defined as follows. Fix a generator $\Omega$ of $\wedge^{\textup{top}}V$, 
\begin{equation}  \label{transfer}
\omega(x, v, v^*) = \chi(v\wedge x v \wedge ... \wedge x^{n-1}v/\Omega).
\end{equation}
Notice that $\omega$ is independent of $v^*$. We have the following theorem:

\begin{theorem}[Existence of Transfers] \label{JR exist}
For any $f \in C_c^\infty(\mathfrak{gl}(V)\times V\times V^*)$, there exists a transfer $\{f_0,f_1\}$. Conversely, for any pair $\{f_0,f_1\}$, there exists a transfer $f$.

\end{theorem}

Now consider Fourier transforms. Fix an additive character $\psi$ of $F$.

In the general linear case, we consider the non-degenerate symmetric bilinear form tr$(XY)$ on $\mathfrak{gl}(V)$. Fix an additive Haar measure $dY$ on $\mathfrak{gl}(V)$. For any $f \in C_c^\infty(\mathfrak{gl}(V)\times V\times V^*)$, we define the patial Fourier transfer on the $\mathfrak{gl}(V)$ factor as

$$\mathcal{F}_1(f)(X,v,v^*)=\int_{\mathfrak{gl}(V)}f(Y,v,v^*)\psi(\textup{tr}(XY))dY.$$

In the unitary case, we consider the non-degenerate symmetric bilinear form $\textup{tr}(XY)$ on $\mathfrak{u}(W)$. We check that tr(XY) takes value in $F$: Take a basis of $W$ and let $H$ denote the Hermitian matrix with respect to this basis. Then $X, Y \in \mathfrak{u}(W)$ implies $H^{-1}X^tH= \overline{X}$ and $H^{-1}Y^tH= \overline{Y}$. Hence  $H^{-1}X^tY^tH= \overline{XY}$. This implies tr$(XY)$ takes value in $F$.

Fix an additive Haar measure $dY$ on $\mathfrak{u}(W)$. For any $f \in C_c^\infty(\mathfrak{u}(W)\times W)$, define

$$\mathcal{F}_1(f)(X,w)=\int_{\mathfrak{u}(W)}f(Y,w)\psi(\textup{tr}(XY))dY.$$

\begin{theorem}[Fourier Transform] \label{JR commute}

Let $n=\textup{dim}V=\textup{dim}W$. There exists a constant $c$ independent of the matching functions, such that if $f$ and $\{f_0,f_1\}$ are transfers, then $c\mathcal{F}_1(f)$ and $\{\mathcal{F}_1(f_0),(-1)^{n-1}\mathcal{F}_1(f_1)\}$ are transfers.

\end{theorem}

The similar theorem for Fourier transforms on the $V\times V^*$ (resp. $W$) factor are also true. But we only need the above theorem in this paper.

\begin{proof}[Proof of theorem $\ref{JR exist}$ and theorem $\ref{JR commute}$]
Both theorem \ref{JR exist} and theorem \ref{JR commute} are proved in \cite{Wei1}. We remark that both proofs are local and do not rely on the fundamental lemma. However, the factor $(-1)^{n-1}$ in theorem \ref{JR commute} is missed in \cite{Wei1}. We shall explain now why there is a $(-1)^{n-1}$ factor. We shall use freely the notation in section 4 of \cite{Wei1}. 

The theorem we cite here is the $B_n$ part of theorem 4.17 in \cite{Wei1}. It is proved along with the $A_n$ part and $C_n$ part by induction. For each $C_n$, there is a $-1$ missing, which by induction, gives $(-1)^{n-1}$ for the $B_n$. The $-1$ is missing from $C_n$ because of the equation (4.30) in \cite{Wei1}. It claims

$$\kappa^\eta(w,w^\prime)=\lambda \kappa(w_i,w_i^\prime).$$

This is proved by reducing it to the identity of Kloosterman sums (theorem 4.12 \cite{Wei1}). However, in this reduction, different $W_i$ gives a $\pm 1$. The correct version is 

$$\kappa^\eta(w,w^\prime)=\pm \lambda \kappa(w_i,w_i^\prime).$$
where one finds $+1$ if $i=0$ and $-1$ if $i=1$.

\end{proof}

Finally, we state the fundamental lemma. However, we will not need this later in this paper. Suppose $E/F$ is unramified. Let $O_E$ and $O_F$ be their rings of intergers.

In the general linear side, take $V=F^n$, consider the characteristic function $1_{\mathfrak{k}}$ of the set M$_n(O_F) \times O_F^n \times O_F^n \subset \mathfrak{gl}(V)\times V\times V^*$.

In the unitary side, take $W_0=E^n$ with the split Hermitian form $\langle \ , \ \rangle$ defined by the anti-diagonal matrix $H$ with all entries $1$.  Consider the characteristic function $1_{\mathfrak{k}_0}$ of the set $\mathfrak{u}(W_0)(O_E) \times O_E^n$. Let $p$ be the residue characteristic of $F$. 
\begin{conjecture}[Fundamental Lemma] \label{Fundamental}
When $p$ is odd, $1_{\mathfrak{k}}$ and $\{1_{\mathfrak{k}_0},0\}$ are transfers.
\end{conjecture}
\begin{remark}
The conjecture is known when $p$ is large enough. This is proved in \cite{Yun} by Z. Yun in the positive characteristic case, extended to characteristic zero by J. Gordon in the appendix to \cite{Yun}. It is not known for small $p$.
\end{remark}

 \section{Endoscopic transfer} \label{LS review}

In this section, we review the endoscopic transfer in unitary Lie algebras.

\subsection{Generality on endoscopy} \label{gen}

The main reference for this section is \cite{Kazhdan}. Let $F$ be a non-archimedean local field of characteristic zero, and $G$ a reductive group over $F$. 

\begin{definition}
An endoscopic triple of $G$ is a triple $(H,s,\eta)$. Here H is a quasi-split group over F, $s$ is a regular semisimple element in $\widehat{G}$, and $\eta$ is an embedding of the complex algebraic groups: $\widehat{H} \hookrightarrow \widehat{G}$. The triple has to satisfy the following conditions$:$
\begin{itemize}
  \item $\eta(\widehat{H})$ is the identity component of the centralizer of $s$ in $\widehat{G}$.
  \item The conjugacy class of $\eta$ is Galois equivariant.
  \item $\eta^{-1}(s)$ is fixed by the Galois action on $\widehat{H}$.
\end{itemize}
\end{definition}

The endoscopic triple are designed to compare $\kappa$-orbit integrals on $G$ with stable orbit integrals on $H$. To state this, we recall the matching of stable conjugacy classes and the transfer of functions. We review this in several steps. In the following, let $(H,s,\eta)$ be an endoscopic triple of $G$. We denote by $\overline{G}$ the base change of $G$ to the algebraic closure of $F$, $G(F)^{r.s.}$ the regular semisimple elements and $\mathfrak{g}$ the Lie algebras of $G$. Similar notations apply to other algebraic groups.

\begin{itemize}
  \item For any embedding of an $F$-torus $T$ into $G$ as  a maximal torus, there is a Galois invariant conjugacy class of embeddings of $\widehat{T}$ into $\widehat{G}$. This is defined by conjugating $\overline{T}$ to $T^{\prime}$ inside $\overline{G}$, where $T^{\prime}$ is a fixed maximal split torus of $\overline{G}$ used to define $\widehat{G}$. The conjugation between $\overline{T}$ and $T^{\prime}$ is not unique, but they all give the same conjugacy class of embeddings: $\widehat{T} \hookrightarrow \widehat{G}$.

  \item Let $T \hookrightarrow H$ and $T \hookrightarrow G$ be two embeddings of an $F$-torus $T$ into $H$ and $G$ as maximal tori. The two embeddings of $T$ match if the corresponding map $\widehat{T} \hookrightarrow \widehat{H}$ and $\widehat{T} \hookrightarrow \widehat{G}$ are conjugate after applying the emmbeding $\eta: \widehat{H} \hookrightarrow \widehat{G}$.
  
  \item Let $\gamma^H \in H(F)^{r.s.}$ and $\gamma \in G(F)^{r.s.}$. Then $\gamma^H$ and $\gamma$ match if there exist an $F$-torus $T$, two matching embeddings $T \hookrightarrow G$ and $T \hookrightarrow H$ and an element $\gamma_0 \in T(F)$, such that $\gamma_0$ maps to $\gamma^H$ and $\gamma$ by the two embeddings.
  
  \item Let $x^H \in \mathfrak{h}(F)^{r.s.}$ and $x \in \mathfrak{g}(F)^{r.s.}$. Then $x^H$ and $x$ match if there exist an $F$-torus $T$, two matching embeddings $T \hookrightarrow G$ and $T \hookrightarrow H$ and an element $x_0 \in \mathfrak{t}(F)$, such that $x_0$ maps to $x^H$ and $x$ by the differentials of the two embeddings.
  
\end{itemize}

\begin{remark}
It is easy to see that the matching between embeddings to maximal F-tori and regular semisimple elements is a matching between stable conjugacy classes in $G$ and $H$. Let $G^*$ be the quasi split inner form of $G$. Every embedding $T \hookrightarrow H$ has a matching $T \hookrightarrow G^*$. This is not true for $G$. An element $x^H \in \mathfrak{h}(F)^{r.s.}$ is $G^*$-regular semisimple if it matches a regular semisimple element $x^*$ in $\mathfrak{g^*}(F)$.
\end{remark}

\begin{definition}
  Let $f^H \in C^{\infty}_c(\mathfrak{h}(F))$ and $f \in C^{\infty}_c(\mathfrak{g}(F))$. Then $f^H$ and $f$ match if the following conditions are satisfied (see below for definitions of these terms)$:$
  
  \textup{1)} For any matching orbits $x^H \in \mathfrak{h}(F)^{G^*-r.s.}$ and $x \in \mathfrak{g}(F)^{r.s.}$, we have the identity:
  \begin{equation} \label{endos}
 \textup{Orb}^{\textup{St}}(f^H,x^H)=\Delta(x^H,x)\textup{Orb}^\kappa(f,x).
 \end{equation}
 
 \textup{2)} If there does not exist $x \in \mathfrak{g}(F)^{r.s.}$ that matches $x^H \in \mathfrak{h}(F)^{G^*-r.s.}$, then 
   \begin{equation} \label{more}
 \textup{Orb}^{\textup{St}}(f^H,x^H)=0.
 \end{equation}
\end{definition}

We explain these terms. The stable orbit integral of $f^H$ at $x^H$ is defined as:  

$$\textup{Orb}^{\textup{St}}(f^H,x^H)=\sum_{x^H_i \thicksim x^H}\textup{Orb}(f^H,x^H_i)$$ 
where $x^H_i$ run through the conjugacy classes in $\mathfrak{h}$ that are stable conjugagte to $x^H$.

Let $T=C_G(x)$ and $\mathfrak{t}$ be the Lie algebra of $T$.  The conjugacy classes in $\mathfrak{g}$ that are stable conjugate to $x$ is a torsor over ker$(H^1(F,T)\rightarrow H^1(F,G))$. Therefore for $x_1$ and $x_2$ that are stable conjugate to $x$, we can define $\textup{inv}(x_1,x_2) \in$ ker$(H^1(F,T)\rightarrow H^1(F,G))$ as the element that maps $x_1$ to $x_2$.

The character $\kappa$ is a character of this group. It is canonically defined by the endoscopic triple $(H,s,\eta)$ via Tate-Nakayama isomorphism. 

\begin{theorem}[Tate-Nakayama Isomorphism \cite{Serre}]
Let F be a non-archimedean local field and T a F-torus. Then the Galois cohomology group $H^1(F,T)$ is canonically isomorphic to the quotient of $$\{x \in X_*(T) | \textup{tr}(x)=0\}$$ by its subgroup generated by $$\{\sigma x-x| \sigma \in \textup{Gal}(\overline{F}/F), x \in X_*(T)\}.$$
\end{theorem}

We now explain how $\kappa$ is defined. The dual group $\widehat{H}$ is equipped with a canonical maximal torus $T_{\widehat{H}}$. The torus $T_{\widehat{H}}$ has a Galois action and $s \in T_{\widehat{H}}$ is invariant under this action. There is an isomorphism $X_*(T) \simeq X^*(T_{\widehat{H}})$ well defined up to conjugation by elements in the Weyl group of $H$. Since $s$ gives a character of $X^*(T_{\widehat{H}})$, via this isomorphism, $s$ also defines a character of $X_*(T)$. The resulting character $\kappa$ is independent of the choice of the isomorphism since $s$ lies in the center of $H$. Now $\kappa$ is trivial on $\{\sigma x-x| \sigma \in \textup{Gal}(\overline{F}/F), x \in X_*(T)\}$ since $s$ is invariant under the Galois action on $T_{\widehat{H}}$ and the Galois actions on $T_{\widehat{H}}$ and $\overline{T}$ differ by an element in the Weyl group of $H$. Finally, by Tate-Nakayama isomorphism, $\kappa$ is a character of $H^1(F,T)$.

The $\kappa$-orbit integrals are then defined as

$$\textup{Orb}^\kappa(f,x)=\sum_{x_i \thicksim x}\textup{Orb}(f,x_i)\kappa(\textup{inv}(x,x_i))$$
where $x_i$ run through the conjugacy classes in $\mathfrak{g}$ that are stable conjugagte to $x$.

Finally, $\Delta(x^H,x)$ is the Langlands-Shelstad transfer factor. It is defined as a function on $\mathfrak{h}^{G^*-r.s} \times \mathfrak{g}^{r.s.}$ such that $\Delta(x,y)$ is not zero if and only if $x$ and $y$ match. The definition of $\Delta(x^H,x)$ is delicate and we will omit here. We only remark that if one change $x$ and $y$ to $x^\prime$ and $y^\prime$ that are stably conjugate, then $\Delta(x,y)$ and $\Delta(x^\prime,y^\prime)$ are changed by multiplying $\kappa(\textup{inv}(y, y^\prime))$. This is evident from equation (\ref{endos}).

We have the following theorem regrading transfer of functions:

\begin{theorem}[Existence of Transfers] \label{LS exist}
For any $f \in C_c^\infty(\mathfrak{g})$, there exists $f_1 \in C_c^\infty(\mathfrak{h})$ matches $f$.
\end{theorem}

Now consider Fourier transforms. Fix an additive character $\psi: F \rightarrow \mathbb{C}^\times$. Fix a $G$-invariant non-degenerate symmetric bilinear form $\langle \ , \ \rangle$ on $\mathfrak{g}$. Define the Fourier transform $\mathcal{F}$ on $C_c^\infty(\mathfrak{g})$ as follows:

$$\mathcal{F}(f)(x)=\int_{\mathfrak{g}}f(y)\psi(\langle x, y \rangle)dy.$$

It can be shown the $G$-invariant bilinear form canonically gives a $H$-invariant non-degenerate symmetric bilinear form on $\mathfrak{h}$, hence a Fourier transform $\mathcal{F}$ on $C_c^\infty(\mathfrak{h})$. (VIII.6 of \cite{Wa1})

\begin{theorem}[Fourier Transform] \label{LS commute}
There exists a constant $c$ independent of the matching functions, such that if $f \in C_c^\infty(\mathfrak{g})$ and $f_1 \in C_c^\infty(\mathfrak{h})$ match, then c$\mathcal{F}(f)$ and $\mathcal{F}(f_1)$ also match.
\end{theorem}
\begin{remark}
The constant $c$ are given explicitly in \cite{Wa2}. It equals the quotient of the Weil indexes of $\mathfrak{h}$ and $\mathfrak{g}$. We will come back to this in section $\ref{JJ}$. 
\end{remark}

\begin{proof}[Proof of theorem $\ref{LS exist}$ and theorem $\ref{LS commute}$]
Both theorem \ref{LS exist} and theorem \ref{LS commute} are deduced in \cite{Wa2} from the fundamental lemma. In our paper, for the case of unitary groups, we are going to deduce theorem \ref{LS exist} (resp. theorem \ref{LS commute})  from theorem \ref{JR exist} (resp. theorem \ref{JR commute}).
\end{proof}

Finally, there are fundamental lemmas assertion characteristic functions of hyperspecial maximal compact match when $H$ and $G$ are unramified. We omit the details since we do not need it in this paper.

\subsection{Stable conjugacy classes in unitary Lie algebras} \label{stable}

In this part, we study stable conjugacy in unitary Lie algebras. 

Let $W$ be a non-degenerate Hermitian space for the extension $E/F$. Recall that we denote by $\mathfrak{u}(W)$ the twisted Lie algebras: 

$$\mathfrak{u}(W)=\{x \in \textup{End}(W)|\langle xu , v\rangle=\langle u , xv\rangle \textup{ for all }u,v \in W \}.$$

Let $\delta$ be a regular semisimple element in $\mathfrak{u}(W).$ The characteristic polynomial of $\delta$ has coefficients in $F$. Let $F[\delta]=\prod_{i=1}^m F_i$, where each $F_i$ is a field. Set $E_i=E \bigotimes_F F_i$, so $E[\delta]=\prod_{i=1}^m E_i$. Finally, let $\{1,...,m\}=S_1\coprod S_2$ with $S_1=\{i|F_i \nsupseteq E\}$.

The following proposition computes the stable conjugacy of $\delta$ in the classical setting.

\begin{proposition}
The centralizer $T_\delta$ of $\delta$ in $\textup{U}(W)$ equals \textup{U}$(1)(E[\delta]^\times/F[\delta]^\times)$. We have $H^1(F, T_\delta)=\prod_{S_1}\mathbb{Z}/2\mathbb{Z}$. And the conjugacy class in the stable conjugacy class of $\delta$ is canonically a torsor of $$\textup{ker}(H^1(F, T_\delta) \rightarrow H^1(F,\textup{U}(W)))= \textup{ker}\left(\prod_{S_1}\mathbb{Z}/2\mathbb{Z} \rightarrow \mathbb{Z}/2\mathbb{Z}\right)$$ where all the map from $\mathbb{Z}/2\mathbb{Z}$ to $\mathbb{Z}/2\mathbb{Z}$ are identity.
\end{proposition}
\begin{proof}

For any $\alpha \in E[\delta]$, let $\overline{\alpha}$ be the conjugation of $\alpha$ by the unique nontrivial automorphism for the extension $E[\delta]/F[\delta]$. The identities $\langle \delta w_1,w_2 \rangle=\langle w_1,\delta w_2 \rangle$ and $\langle cw_1,w_2 \rangle=\langle w_1,\overline{c}w_2 \rangle$ ($c \in E$) together imply that for any $\alpha \in E[\delta]$, $ \langle \alpha w_1, w_2 \rangle= \langle w_1,\overline{\alpha}w_2 \rangle$.

Since Centralizer of $\delta$ in $\textup{GL}(W)$ equals $E[\delta]^\times$,  centralizer of $\delta$ in $\textup{U}(W)$ equals $\{\alpha \in E[\delta]^\times| \langle \alpha w_1, \alpha w_2 \rangle = \langle w_1,w_2 \rangle\}$. But $\langle \alpha w_1, \alpha w_2 \rangle=\langle \overline{\alpha} \alpha w_1,  w_2 \rangle$. Hence the centralizer equals U$(1)(E[\delta]^\times/F[\delta]^\times)$.

Now, a simple computation shows $H^1(F,\textup{U}(1)(E_i/F_i))$ is trivial if $F_i \supseteq E$ and equals $\mathbb{Z}/2\mathbb{Z}$ if $F_i \nsupseteq E$. Therefore $H^1(F, T_\delta)=\prod_{S_1}\mathbb{Z}/2\mathbb{Z}$. The rest of the proposition follows from standard results on stable conjugacy.

\end{proof}

In our setting, it is more natural to consider not only the conjugacy class in $\mathfrak{u}(W)$, but also its Jacquet-Langlands transfer on the other Hermitian space. So let $W_0$ and $W_1$ be the representatives of the two isomorphism classes of Hermitian spaces of dimension $n$. For $\delta_i \in \mathfrak{u}(W_i)$, we shall say $\delta_0$ and $\delta_1$ are Jacquet-Langlands transfers if they become conjugate in the algebraic closure $\overline{F}$ of $F$. When $\delta_0$ and $\delta_1$ are transfers, they are conjugate over $\overline{F}$. This conjugation defines an isomorphism between $F[\delta_0]$ and $F[\delta_1]$, which is independent of the conjugation we choose. In particular $T_{\delta_0}$ and $T_{\delta_1}$
are canonically isomorphic.

We warn the reader that even when $n$ is odd, so $\mathfrak{u}(W_0)$ and $\mathfrak{u}(W_1)$ are abstractly isomorphic, we still distinguish them as two different Lie algebras. In fact, they are different as pure inner forms.

Now fix a stable conjugacy class $\delta \in \mathfrak{u}(W_0)$.  Let $S$ be the union of the conjugacy classes in $\mathfrak{u}(W_0)$ that are stable conjugate to $\delta$ and the conjugacy classes in $\mathfrak{u}(W_1)$ that are Jacquet-Langlands transfers of $\delta$. For any $\delta^' \in S$, let $T$ be the centralizer of $\delta^'$.  Different choices of $\delta^'$ define $T$ that are canonically isomorphic to each other.

We want to show $S$ is naturally a torsor over $H^1(F,T)$. We start with the case when $m=1$, namely $F[\delta]=F_1$ is a field.

\begin{proposition}
Let $m=1$. When $F_1 \supseteq E$, then \textup{dim}$W$ is even and $S$ is a single point belongs to the split Hermitian space. When $F_1 \nsupseteq E$, then $S$ consists of two points, with one point for each Hermitian space.
\end{proposition}

\begin{proof}
Let $E_1=E \bigotimes_F F_1$.
To find whether S contains element in $\mathfrak{u}(W_i)$, we need to compute the determinant of any Hermitian structure on $E_1$ such that $\delta \in \mathfrak{u}(E_1)$. Given such an Hermitian form $\langle \ , \ \rangle$ and using $\delta \in \mathfrak{u}(E_1)$, we find $\langle x , y \rangle=\text{tr}_{E_1/E}(ax\overline{y})$ for some $a \in F_1^\times$. If $a_1$ and $a_2$ differ by multiplication of elements in Nm($E_1^\times/F_1^\times$), then the Hermitian space are isomorphic. So we only need to consider $a \in F_1^\times/\text{Nm}(E_1^\times/F_1^\times)$.

When $F_1 \supseteq E$, then $F_1^\times/\text{Nm}(E_1^\times/F_1^\times)=\{1\}$ and the Hermitian space is unique. It has to be the split one since only the quasi-split unitary Lie algebras contain elements that are not Jacquet-Langlands transfer from its non-trivial inner form. This can also be verified directly.

When $F_1 \nsupseteq E$, then $F_1^\times/\text{Nm}(E_1^\times/F_1^\times)=\mathbb{Z}/2\mathbb{Z}$. The determinant of $\text{tr}_{E_1/E}(ax\overline{y})$ equals Nm$_{F_1/F}(a)$ times a constant. By class field theory, this surjects to $F^\times/\text{Nm}_{E/F}(E^\times)$. We conclude that both Hermitian spaces occur.

Finally, the fact that $S$ intersects with each $\mathfrak{u}(W_i)$ with at most one point follows from the computation of the conjugacy class in a fixed Lie algebras as $m=1$.

\end{proof}

We can now treat the general case.
 
 \begin{proposition} \label{bijection}
 
 There is a natural isomorphism $S \rightarrow \prod_{S_1}\mathbb{Z}/2\mathbb{Z}$. It defines $S$ as a torsor over $H^1(F,T)$ which extends the classical torsor structure over $ker(H^1(F,T) \rightarrow H^1(F,U(W))$.
 
\end{proposition}
\begin{proof}

The map is defined in the following way. Let $\delta \in S$ belongs to $\mathfrak{u}(W)$. As usual, $F[\delta]=\prod F_i, E_i=E \bigotimes_F F_i$ and $W=\bigoplus W_i$, where $W_i=E_iW$. Since $\delta$ is regular semisimple and $\delta \in \mathfrak{u}(W)$, these $W_i$ are orthogonal to each other. By the $m=1$ case, the Hermitian form on $W_i$ is unique if $F_i \supseteq E$ and there are two possible choices if $F_i \nsupseteq E$. Therefore the type of these $W_i$ defines a map $S \rightarrow \prod_{S_1}\mathbb{Z}/2\mathbb{Z}$. Here, we normalize so that the split Hermitian space maps to 0.

The map $S \rightarrow \prod_{S_1}\mathbb{Z}/2\mathbb{Z}$ admits a section. Indeed, for any element in $\mathbb{Z}/2\mathbb{Z}$, we may construct a Hermitian space $W_i$ as in the $m=1$ case. Then $W=\bigoplus W_i$ is a Hermitian space such that  $F[\delta] \subset \mathfrak{u}(W)$. The space $W$ is isomorphic to one of the Hermitian spaces of dimension $n$. And any such isomorphism defines an element in $S$ that is independent of the isomorphism we take. This implies the map $S \rightarrow \prod_{S_1}\mathbb{Z}/2\mathbb{Z}$ is surjective. And by counting numbers, we find it is an isomorphism. It also implies that $S$ contains elements of both Hermitian spaces except when $S_1=\emptyset$.

This defines $S$ as a torsor over $H^1(F,T)$. Notice that $\delta_1$ and $\delta_2 \in S$ are in the same Hermitian space exactly when the corresponding $W_i$ differs at even number of places. Now it is almost trivial to check the torsor structure agrees with the classical one.
\end{proof}

The proposition shows elements in $S$ are uniquely determined by the type of these $W_i$. It also allows us to define the invariant of two elements in $S$:
\begin{definition}
Let $\rho$ denote the map $S \rightarrow H^1(F,T)$. Given $\delta_1, \delta_2 \in S$, define $$\textup{inv}(\delta_1, \delta_2)=\rho(\delta_1)-\rho(\delta_2).$$

\end{definition}
  
\subsection{Endoscopic transfer for unitary Lie algebras}

In this section, we consider the endoscopic transfer in the case of unitary Lie algebras. 

Let $W$ be a non-degenerate Hermitian space of dim $n$ for $E/F$. In $\overline{F}$, $\overline{\textup{U}(W)}$ corresponds to the root datum:

$$(\mathbb{Z}^n,\mathbb{Z}^n,\alpha_i-\alpha_j,\lambda_i-\lambda_j)$$
where the weight $\alpha_i$ maps $(x_1,...,x_n)$ to $x_i$ and the coweight $\lambda_i$ maps $t$ to $(1,1,...,t,...,1)$ with $t$ in the $i$-th component. The root datum admit an action of the Gal$(\overline{F}/F)$ which factors through Gal($E/F$). The nontrivial element $\sigma \in$ Gal($E/F$) acts as follows: $\sigma(\lambda_i)=-\lambda_{n+1-i}$ and $\sigma(\alpha_i)=-\alpha_{n+1-i}$. Hence $\widehat{U(W)} \cong GL_n(\mathbb{C})$ with $\sigma \in$ Gal($E/F$) acting as an outer isomorphism on GL$_n(\mathbb{C})$ (when $n \geqq 3$).

The endoscopic triples $(H, s, \eta)$ are constructed as follows. Let $n=a+b$. Let $W_k$ be the Hermitian space of dimension $k$ with Hermitian matrix that is anti-diagonal with entries $1$ and $ \textup{U}(k) \cong \textup{U}(W_k)$. So U$(a)$ is the quasi split unitary group of dimension $a$ for $E/F$. Take $H=\textup{U}(a) \times \textup{U}(b)$. The embedding $\eta$ is the natural embedding of $GL_{a}(\mathbb{C}) \times GL_{b}(\mathbb{C})$ into $GL_n(\mathbb{C})$. Finally, $s=(1,1,...,1,-1,...,-1)$ with the first $a$ coordinates equal $1$. It is easy to verify these $(H, s, \eta)$ are indeed endoscopic triples of U$(W)$. In fact, they are all the possible elliptic endoscopic groups.

Let us write down the character $\kappa$ in this case. Given matching orbits $\delta$ and $(\delta_1,\delta_2)$, the character $\kappa$ is a character of $H^1(F,T_\delta)=H^1(F,T_{\delta_1}\times T_{\delta_2})=H^1(F,T_{\delta_1}) \times H^1(F,T_{\delta_2})$ where $T_\delta$ is the centralizer of $\delta$. Recall $F[\delta]=\prod F_i$ and $S_1=\{i|F_i \nsupseteq E\}$. Let us define for $k=1$ or $2$: $$S_1(\delta_k)=\{i \in S_1|F_i \textup{ comes from } \delta_k\}.$$

Since $s=(1,1,...,1,-1,...,-1)$, Tate-Nakayama duality implies:

\begin{proposition} \label{definition of kappa}
The character $\kappa$ of $H^1(F,T_\delta)$ is defined as follows: on each copy of $\mathbb{Z}/2\mathbb{Z}$ corresponding to $S_1(\delta_1)$, $\kappa$ is the trivial map. On each copy of $\mathbb{Z}/2\mathbb{Z}$ corresponding to $S_1(\delta_2)$, $\kappa$ is the unique nontrivial map.
\end{proposition}

Finally, we mention the transfer factor. To this end, for each matching stable conjugacy orbits  $\delta$ and \textup{(}$\delta_1,\delta_2$\textup{)}, we need to pick a conjugacy class in the stable conjugacy class of $\delta$. We distinguish two cases.

\textbf{Case 1.} $W \simeq W_{a} \bigoplus W_{b}$. This is the case for example when $E/F$ is unramified and $W=W_n$. Choose an isomorphism between $W$ and $W_{a}\bigoplus W_{b}$. This defines an embedding of $\mathfrak{u}(W_{a}) \bigoplus \mathfrak{u}(W_{b})$ into $\mathfrak{u}(W)$, which is well defined up to conjugate by U$(W)$. For any $(\delta_1,\delta_2) \in \mathfrak{u}(W_{a}) \bigoplus \mathfrak{u}(W_{b})$, let $\delta \in \mathfrak{u}(W)$ be its image under the embedding.

If $(\delta_1,\delta_2)$ is taken in a stable conjugacy class, then $\delta$ is in a single stable conjugacy class. Furthermore, $\delta$ matches $(\delta_1,\delta_2)$ in the endoscopic sense. This is because stable conjugacy of regular semisimple elements in $\mathfrak{u}(W)$ is determined by characteristic polynomials and $\delta$ has the same characteristic polynomial with $(\delta_1,\delta_2)$. Also, if $(\delta_1,\delta_2)$ is taken in a conjugacy class, then $\delta$ is in a single conjugacy class (notice that the embedding is well defined up to conjugate over $F$). This means we have a preferable choice of a conjugacy class in the endoscopic transfer of $(\delta_1,\delta_2)$. This is useful since two different conjugacy classes in U$(W)$ will change the transfer factor in (\ref{endos}) by their relative position.

\begin{definition} \label{nice}
We define elements $\delta$ and \textup{(}$\delta_1,\delta_2$\textup{)} are nice matching elements if they arise in the way described above.
\end{definition}

For nice matching orbits $\delta$ and $(\delta_1,\delta_2)$, define 

$$D(\delta)=\prod_{x_1,x_2}(x_1-x_2)$$ where $x_1$ (resp. $x_2$) run through all the eigenvalues of $\delta_1$ (resp. $\delta_2$) in $\overline{F}$.
Then the transfer factor equals

$$\Delta((\delta_1,\delta_2),\delta)=\chi(D(\delta))|D(\delta)|_F.$$

For example, when $E/F$ is unramified, the transfer factor equals $(-q)^{-r}$ with $r=v_F(D(\delta))$.

\textbf{Case 2.}  $W \ncong W_{a} \bigoplus W_{b}$. Let $\delta$ and \textup{(}$\delta_1,\delta_2$\textup{)} are matching elements. Choose $\delta^\prime$ such that $\delta^\prime$ and \textup{(}$\delta_1,\delta_2$\textup{)} are nice matching elements. So $\delta^\prime$ arises from the other Hermitian space. Then 

$$\Delta((\delta_1,\delta_2),\delta)=\kappa(\textup{inv}(\delta,\delta^\prime)) \chi(D(\delta))|D(\delta)|_F .$$

\subsection{Jacquet-Langlands transfer for unitary Lie algebras} \label{JJ}

Let $W_0$ be the split Hermitian space of dimension $n$ and $W_1$ the non-split Hermitian spaces of dimension $n$. The reductive group U$(W_0)$ is an endoscopic group of U$(W_1)$, whose endoscopic triple is $(\textup{U}(W_0), 1, \textup{id})$. Hence everything we have defined for endoscopic transfer applies to this case. Endoscopic transfers of orbits are the same as Jacquet-Langlands transfers. The transfer factor is identically $1$.

Then theorem \ref{LS exist} and theorem \ref{LS commute} in this case mean:
\begin{theorem} \label{JL exist}
For any $f_1 \in C_c^\infty(\mathfrak{u}(W_1))$, there exists a transfer $f_0  \in C_c^\infty(\mathfrak{u}(W_0))$.
\end{theorem}

\begin{theorem} \label{JL commute}
Let $f_i \in C_c^\infty(\mathfrak{u}(W_i))$. There exists a constant $c$ independent of the matching functions, such that if $f_0$ and $f_1$ match, then $\mathcal{F}(f_0)$ and c$\mathcal{F}(f_1)$ also match.
\end{theorem}

\begin{proof}[Proof of theorem \ref{JL exist} and theorem \ref{JL commute}]
Both theorem \ref{JL exist} and theorem \ref{JL commute} are proved in \cite{Wa2} (theorem 1.5 in \cite{Wa2}). Since the fundamental lemmas between inner forms are trivial, the proof in \cite{Wa2} applies without any assumption.
\end{proof}

Theorem \ref{JL exist} and theorem \ref{JL commute} will be used in our proof of the main theorem. We also need the explicit constant $c$ in (\ref{JL commute}). Recall the definition of Weil index:
\begin{definition}[Weil index]
Let $F$ be a non-archimedean local field of characteristic zero and $V$ a vector space of dimension $n$ over $F$ equipped with a quadratic form $q: V \rightarrow F$. Fix an additive character $\psi: F \rightarrow \mathbb{C}^\times$. Let $f_q$ be the function $\psi \circ q$ and $\mathcal{F}$ the Fourier transform defined by $\psi$. There exists a constant $\gamma_q(\psi)$, called Weil index, such that for all $f \in C_c^\infty(V)$, the following identity is true:
$$\mathcal{F}(f_q)=\gamma_q(\psi)f_{-q}.$$
\end{definition}

Waldspurger shows in \cite{Wa2} the following:
\begin{proposition}
The constant $c$ in theorem $\ref{LS commute}$ equals $\frac{\gamma_{q_G}(\psi)}{\gamma_{q_H}(\psi)}$. Here $q_G$ and $q_H$ are quadratic forms on $\mathfrak{g}$ and $\mathfrak{h}$ that match each other.
\end{proposition}

Let us compute this quotient of Weil index in the case of unitary Lie algebras:
\begin{proposition} \label{E1}
The constant $c$ in theorem $\ref{JL commute}$ equals $(-1)^{n-1}$.
\end{proposition}
\begin{proof}
We only treat the case when $n=2$. The general case is $n-1$ copy of the $n=2$ case plus a common part in $\mathfrak{u}(W_0)$ and $\mathfrak{u}(W_1)$. Consider the $2$ dimensional Hermitian space $W$ for $E/F$ with Hermitian matrix diag$\{1,\lambda\}$. Choose the non-degenerate quadratic form as $q_W=$tr($X^2$). A basis of $\mathfrak{u}(W)$ is given by 

$$ \begin{pmatrix}
    1 & 0  \\
    0 & 0 
  \end{pmatrix},
   \begin{pmatrix}
    0 & \lambda  \\
    1 & 0 
  \end{pmatrix},
   \begin{pmatrix}
    0 & -\varepsilon \lambda  \\
    \varepsilon & 0 
  \end{pmatrix},
   \begin{pmatrix}
    0 & 0  \\
    0 & 1 
  \end{pmatrix}$$  
where $\varepsilon^2=x \in F$.

Let $q_a$ denote the quadratic form on $F$ of the form $ax^2$, then this basis defines $q_W \cong q_1 \bigoplus q_{2\lambda} \bigoplus q_{-\varepsilon \lambda} \bigoplus q_1$. Let $\gamma_a=\gamma_{q_a}(\psi)$, then $\gamma_{q_W}=\gamma_1^2\gamma_{2\lambda} \gamma_{-2\varepsilon \lambda}$ since Weil index commutes with direct sum. Suppose $\gamma_{q_W}$ is independent of $\lambda$, this implies:
\begin{equation} \label{A1}
\gamma_{a}\gamma_{-xa}=\gamma_1\gamma_{-x}
\end{equation}
where $a \in F^\times$ and is not a norm from $E^\times$. We have the following properties of Weil index:
$$\gamma_{-a}=\gamma_{a}^{-1},$$
$$\gamma_a\gamma_b=\gamma_1\gamma_{ab}(a,b)_F$$
where $(a,b)_F$ is the Hilbert symbol.

Using $(a,x)_F=-1$ and the above two identities, one show (\ref{A1}) is only correct if we multiple one side by $-1$. This means $c=-1$.

\end{proof}

\section{Nilpotent identity} \label{Nil}
The goal of this section is to prove theorem \ref{Nilpotent Identity}. It gives an identity in the Jacquet-Rallis transfer between certain nilpotent orbit integrals. In the unitary side, this nilpotent orbit integrals turns out to be the usual $\kappa$-orbit integrals in the endoscopic case. So theorem \ref{Nilpotent Identity} enables us to reduce the endoscopic transfer identity to the general linear side. 

As indicated in the introduction, theorem \ref{Nilpotent Identity} is proved by approaching the nilpotent orbits by regular semisimple ones and using germ expansion identities. Here, germ expansion identities are relations between orbit integrals of nilpotent orbits and very close regular semisimple orbits. We break the proof into several parts. We first prove germ expansion identities for tori. We use this to prove germ expansion identities in the two sides of the Jaquet-Rallis transfer. Then we relate the germ expansion identities in the two sides and finally prove theorem \ref{Nilpotent Identity}.

\subsection{Germ expansions in the torus case} \label{A lemma}

Let $E/F$ be a quadratic extension of non-archimedean local fields, $\chi$ the quadratic character of $F^\times$ corresponding to $E/F$ and $F_i/F$ finite extensions of fields.  We set

$$A=\prod_1^m F_i\quad,\quad T=\prod_1^m F_i^\times.$$

When $a \in A$ and $\Lambda \subset \{1,...,m\}$, we denote by $a_\Lambda$ the element in $A$ which equals $a$ in the $\Lambda$ component and 0 otherwise. Let $q$ be the number of elements in the residue field of $F$, we have the norm map $|\cdot|_{i}: T \xrightarrow{\pi_i} F_i^\times \rightarrow \mathbb{R}^\times_{>0}$ extending the given one on $F^\times$ (normalized so that $|\pi_F|=\frac{1}{q}$). We define $|\cdot|_\Lambda: T \rightarrow \mathbb{R}^\times_{>0}$ to be the products of $|\cdot|_{i}$ for $i \in \Lambda$. Finally, Write $\{1,...,m\}=S_1 \amalg S_2$, where $S_1$ contains all $i$ such that $F_i \nsupseteq E$ and $S_2$ contains all $i$ such that $F_i \supseteq E$. 

\begin{lemma}[Germ Expansions for Tori] \label{main}
Let $f \in C_c^\infty(A\times A)$ and  consider the following integral $$\textup{Orb}(f,\varepsilon)=\int_T f(t, \varepsilon t^{-1})\chi(t)dt.$$ 

When $\varepsilon \in T$ is close enough to 0 (as an element in $A$), the above integral equals:

$$\sum_{\Lambda_2 \subseteq S_2}(-1)^{|S_2 \backslash \Lambda_2|} c_{\Lambda_2}(f) \log_q(|\varepsilon|_{S_2 \backslash \Lambda_2})$$
where $c_{\Lambda_2}(f)$ is independent of $\varepsilon_{S_2}$ and 

$$c_\emptyset(f)=\sum_{\Lambda_1 \subseteq S_1} \int_{T_{S_1}} f(t_{\Lambda_1},(t^{-1})_{S_1 \backslash \Lambda_1}) \chi(t) \prod_{i \in S_1\backslash \Lambda_1} \chi(\varepsilon_i) dt .$$

The integral in $c_\emptyset$ is defined by analytic continuation.


\end{lemma}

To explain this lemma, consider the action of $T$ on $A \times A$ by $(t, t^{-1})$. Then $A$ will be the quotient space and $\pi: A \times A \rightarrow A$ the quotient map, $\pi(x,y)=xy$.
The integrals $\textup{Orb}(f,\varepsilon)$ are orbit integrals on regular semisimple orbits whose invariants approach 0 as $\varepsilon \rightarrow 0$. By the general principle of Shalika germs, it can be expressed as a summation over ($T,\chi$) invariant distributions supported on $\pi^{-1}(0)$ with coefficients depending on the regular semisimple orbits. 

Our lemma gives an explicit formula in this case, where we first arrange these ($T,\chi$) invariant distributions by $\Lambda_2 \subseteq S_2$ and each $\Lambda_2$ is distinguished by the term $\log_q(|\varepsilon|_{S_2 \backslash \Lambda_2})$. Then for the part $\Lambda_2=\emptyset$, we give an explicit formula for $c_\emptyset(f)$. The term $c_\emptyset(f)$ contains $2^{|S_1|}$ different $(T,\chi)$ invariant distributions supported on $\pi^{-1}(0)$ indexed by $\Lambda_1 \subseteq S_1$. For each $\Lambda_1$, the coefficient of the corresponding distribution is  $\prod_{i \in S_1 \backslash \Lambda_1} \chi(\varepsilon_i)$. One can also give a general formula for all the $c_\Lambda(f)$, but this is not needed here.

\begin{proof}

Both sides are linear in $f$, so we can assume $f=\prod f_i$ with $f_i \in C_c^\infty (F_i \times F_i)$. Using this $f$ to substitute reduces the question to the case when $m=1$. So we can assume $m=1$.

Now distinguish two cases. In the first case, $F_i \supseteq E$, so $\chi$ is trivial, we want to prove:
$$\int_{F_i^\times}f(t,\varepsilon t^{-1})=c-f(0,0)\log_q|\varepsilon|.$$

Here $c$ is independent of $\varepsilon$ and in fact 

$$c=\bigg(\int_{F_i^\times}f(t,0)|t|^s+\int_{F_i^\times}f(0,t^{-1})|t|^s\bigg)\bigg\rvert_{s=0}.$$

The integrals $\int_{F_i^\times}f(t,0)|t|^s$ and $\int_{F_i^\times}f(0,t^{-1})|t|^s$ are understood as their analytic continuations. Both integrals have a simple pole at $0$ with residues that are negative of each other, so there sum is holomorphic at $0$.

To prove the identity, notice that the identity is linear in $f$. If $f(x,0) \equiv 0$ or $f(0,y) \equiv 0$ or $f=1_{O_F \times O_F}$,  the identity is easily verified. Now the identity follows because $C_c^\infty(A \times A)$ is generated by those three kinds of functions.

For the other case, when $F_i \nsupseteq E$, we want to prove:

$$\int_{F_i^\times}f(t,\varepsilon t^{-1})\chi(t)=\int_{F_i^\times}f(t,0)\chi(t)|t|^s\bigg\rvert_{s=0}+\int_{F_i^\times}f(0,\varepsilon t^{-1})\chi(t)|t|^{-s}\bigg\rvert_{s=0}.$$

In this case both integrals converge for Re$(s) > 0$, extends to a meromorphic function on the whole complex plane and is holomorphic at $0$. The identity can be proved in the same way as in the previous case.

\end{proof}

\subsection{Germ expansions in the general linear case} \label{general linear germ}

Let $V$ be a vector space over $F$. Choose regular semisimple $(\gamma, v, v^*) \in \mathfrak{gl}(V) \times V \times V^*$. We also assume $\gamma$ is regular semisimple.

 Let $F[\gamma]=\prod_{i=1}^m F_i$ and identify $\gamma$ with $\prod \gamma_i$. The fact that $\gamma$ is regular semisimple means the pairs $(F_i, \gamma_i )$ are non-isomorphic to each other. Let $V_i=F[\gamma_i]V$, so $V=\bigoplus V_i$. As before, set $S_1=\{i | F_i \nsupseteq E\}$ and $S_2=\{i | F_i \supseteq E\}$. 

 Let $f \in  C_c^\infty (\mathfrak{gl}(V)\times V\times V^*)$ and $\varepsilon \in F[\gamma]^\times$.  In section \ref{JR review}, we have defined the following orbit integral: 
 
 $$\textup{Orb}(f,(\gamma,v,v^*\varepsilon)) = \int_{\textup{GL}(V)}f(g\gamma g^{-1},gv,v^*\varepsilon g^{-1})\chi(g)dg.$$
 
 \begin{proposition} [Germ Expansions]
 
When $\varepsilon$ is close enough to $0$ \textup{(}depending on $f$\textup{)}$:$

$$\textup{Orb}(f,(\gamma,v,v^*\varepsilon))=\sum_{\Lambda_2 \subseteq S_2}(-1)^{|S_2 \backslash \Lambda_2|} c_{\Lambda_2}(f) \textup{log}_q(|\varepsilon|_{S_2 \backslash \Lambda_2})$$
with $c_{\Lambda_2}(f)$ independent of $\varepsilon_{S_2}$ and 
 
 $$c_\emptyset(f)=\sum_{\Lambda_1 \subseteq S_1}  \textup{Orb}(f,(\gamma,v_{\Lambda_1},v_{S_1 \backslash \Lambda_1}^*))\prod_{S_1 \backslash \Lambda_1}\chi(\varepsilon_i).$$
 
In the definition of $c_\emptyset(f)$, we use ``nilpotent orbit integrals'' that are defined in \textup{(\ref{converge})}.
 \end{proposition}
 
 Once again,  the `nilpotent orbit integrals'' in $c_\emptyset$ are $(G,\chi)$ invariant distributions on the fiber $\pi^{-1}(\text{inv}(\gamma),0)$. Here we consider the points on the quotient space where $a_i$ is defined by $\gamma$ and $b_i=0$. Because of  Shalika germ, it is natural to expect that they appears in the formula.

 \begin{proof}
 
$ \textup{Orb}(f,(\gamma,v,v^*\varepsilon))$ equals $$\int_{\textup{GL}(V)/T_\gamma} \bigg(\int_{T_\gamma}f(g\gamma g^{-1},gtv,v^*\varepsilon t^{-1}g^{-1})\chi(t)\chi(g)dt\bigg)dg$$
where we denote by $T_\gamma$ the centralizer of $\gamma$.
 
 For any fixed $g$, we may apply lemma \ref{main} to the function 
 
 $$f_g(x,y) = f(g\gamma g^{-1},gxv,v^*y g^{-1}).$$
 
 We claim that when $\varepsilon$ is close enough to $0$, lemma \ref{main} will work for all the $f_g$. This is because $\gamma$ is regular semisimple, the outer integral is over a compact set and the map $g \mapsto f_g$ is locally constant.
 
Hence when $\varepsilon$ is small enough, lemma \ref{main} implies that the orbit integral equals
 
 $$\sum_{\Lambda_2 \subseteq S_2}(-1)^{|S_2 \backslash \Lambda_2|} c_{\Lambda_2}(f) \textup{log$_q$}(|\varepsilon|_{S_2 \backslash \Lambda_2}) $$
with $c_{\Lambda_2}(f)$ independent of $\varepsilon_{S_2}$ and

 $$ c_\emptyset(f)=\sum_{\Lambda_1 \subseteq S_1} \int_{\textup{GL}(V)/T_\gamma} \bigg(\int_{(T_\gamma)_1} f(g \gamma g^{-1}, gt_{\Lambda_1}v,v^*t^{-1}_{S_1 \backslash \Lambda_1}g^{-1})\chi(t)\chi(g)dt\bigg)d\overline{g} \prod_{S_1 \backslash \Lambda_1}\chi(\varepsilon_i).$$
 
 This is exactly
 
 $$\sum_{\Lambda_1 \subseteq S_1}  \textup{Orb}(f,(\gamma,v_{\Lambda_1},v_{S_1 \backslash \Lambda_1}^*))\prod_{S_1 \backslash \Lambda_1}\chi(\varepsilon_i).$$
 
 \end{proof}
 
 \subsection{Germ expansions in the unitary case} \label{unitary germ}

Let $W$ be a Hermitian space over $E$. Choose regular semisimple $(\delta, w) \in \mathfrak{u}(W) \times W$. We also assume $\delta$ is regular semisimple. 

The characteristic polynomial of $\delta$ has coefficients in $F$. Write $F[\delta]=\prod_{i=1}^m F_i$ and $E_i=F_i \bigotimes_F E$. The vector space $W$ is canonically isomorphic to $\bigoplus E_i$.  As before, set $S_1=\{i | F_i \nsupseteq E\}$, $S_2=\{i | F_i \supseteq E\}$.

Let $f \in C_c^\infty (\mathfrak{u}(W) \times W)$ and $\varepsilon \in E[\delta]^\times$. In section \ref{JR review}, we have defined the following orbit integral: 

$$\textup{Orb}(f,(\delta, \varepsilon w)) = \int_{\textup{U}(W)}f(g\delta g^{-1},g\varepsilon w)dg.$$

\begin{proposition}[Germ Expansions]
When $\varepsilon$ is close enough to $0$ \textup{(}depending on f\textup{)}$:$
$$\textup{Orb}(f,(\delta, \varepsilon w))=\sum_{\Lambda_2 \subseteq S_2}(-1)^{|S_2 \backslash \Lambda_2|} \cdot c_{\Lambda_2}(f) \log_q(|{\varepsilon}|_{S_2 \backslash \Lambda_2})$$
with $c_{\Lambda_2}(f)$ independent of $\varepsilon_{S_2}$ and 
 
 $$c_\emptyset(f)=\int_{\textup{U}(W)/T_\delta}f(g\delta g^{-1},0)d\overline{g}.$$
 
\end{proposition}
\begin{proof}
$\textup{Orb}(f,(\delta, \varepsilon w))$ equals

$$\int_{\textup{U}(W)/T_\delta} \int_{T_\delta}f(g\delta g^{-1},gt\varepsilon w)dtd\overline{g}$$
where $T_\delta= \textup{U}(1)(E[\delta]^\times/F[\delta]^\times)$ is the centralizer of $\delta$. 

Let $\varepsilon=(\varepsilon_i) \in E[\delta]^\times$ be close enough to $0$. When $i \in S_1$, U$(1)(E_i/F_i)$ is compact. So when $\varepsilon_i \rightarrow 0$ for $i \in S_1$, the $i$ component of $t\varepsilon w$ tends to $0$ uniformly. When $i \in S_2$,  we may identify $E_i$ with $F_i \times F_i$ and let $\varepsilon_i \in E_i$ correspond to $(\varepsilon_i^{(1)},\varepsilon_i^{(2)})$ through this identification. 

Because $\delta$ is regular semisimple, the integral over $g$ is over an compact set. Recall for $t \in T$, we denote by $t_\Lambda$ the element which equals to $t$ for the components in $\Lambda$ and equals $0$ otherwise. When $\varepsilon$ is close enough to $0$,

$$f(g\delta g^{-1},gt\varepsilon w)=f(g\delta g^{-1},g(t\varepsilon)_{S_2} w).$$ 

Therefore
\begin{align*}
\textup{Orb}(f_W,(\delta, \varepsilon w))
  &   =\int_{\textup{U}(W)/T_\delta} \int_{(T_\delta)_2}f(g\delta g^{-1},g(t\varepsilon)_{S_2} w)dtd\overline{g} \\
 & =\int_{\textup{U}(W)/T_\delta} \int_{(T_\delta)_2}f(g\delta g^{-1},g(t\varepsilon_i^{(1)},t^{-1}\varepsilon_i^{(2)}) w)dtd\overline{g} \\
 &=\int_{\textup{U}(W)/T_\delta} \int_{(T_\delta)_2}f(g\delta g^{-1},g(t,t^{-1}\varepsilon_i^{(1)}\varepsilon_i^{(2)}) w)dtd\overline{g}.
\end{align*}

Now, apply lemma \ref{main} to the function 

$$f_g(x,y)=f(g\delta g^{-1},g(x,y)w).$$

As before, we can find $\varepsilon$ very close to $0$ that works for all the $f_g$. Hence the orbit integral equals

$$\sum_{\Lambda_2 \subseteq S_2}(-1)^{|S_2 \backslash \Lambda_2|} \cdot c_{\Lambda_2}(f) \log_q(|{\varepsilon}|_{S_2 \backslash \Lambda_2})$$
with $c_{\Lambda_2}(f)$ independent of $\varepsilon_{S_2}$ and

 $$c_{\emptyset}(f)=\int_{\textup{U}(W)/T_\delta}f(g\delta g^{-1},0)d\overline{g}.$$

\end{proof}

\subsection{Matching in the Jacquet-Rallis transfer} \label{section Jacquet-Rallis} 

We combine the previous two germ expansions together. Let $V$ be a finite dimensional vector space over $F$. Let $W_0$ and $W_1$ be representatives of the two isomophism classes of finite dimensional Hermitian spaces over $E$ of the same dimension as $V$. 

Choose matching orbits $(\gamma, v, v^*)$ and $(\delta, w)$. We also assume that $\gamma$ (resp. $\delta$) is regular semisimple in $\mathfrak{gl}(V)$ (resp. $\mathfrak{u}(W)$). Since $\gamma$ and $\delta$ have the same characteristic polynomials and are regular semisimple, they can be conjugated over $F^{\textup{alg}}$. This defines an isomorphism $F[\gamma] \cong F[\delta]$ which is independent of the conjugation we choose.

\begin{lemma} \label{change match}

When $(\gamma, v, v^*)$ and $(\delta, w)$ match, choose $\varepsilon \in F[\gamma]^\times$, $\varepsilon_1 \in E[\delta]^\times$ with$$ \textup{Nm}_{E[\delta]/F[\delta]}(\varepsilon_1)=\varepsilon$$ then $(\gamma, v,  v^*\varepsilon)$ and $(\delta, \varepsilon_1 w)$ also match. 

\end{lemma}
\begin{proof}
The invariants $a_i$ are equal since they only depend on $\gamma$ and $\delta$. For the invariants $b_i$, since $(\gamma, v, v^*)$ and $(\delta, w)$ match, we have $\langle v^*\varepsilon, \gamma^iv\rangle=\langle v^*, \gamma^i \varepsilon v\rangle=\langle w, \delta^i \varepsilon w \rangle=\langle w, \delta^i \varepsilon_1\overline{\varepsilon_1}w \rangle=\langle \varepsilon_1w, \delta^i \varepsilon_1w \rangle$.

\end{proof}

Let $f$ and \{$f_W$\} match, so
 $$\textup{Orb}(f_W,(\delta, \varepsilon_1 w))=\textup{Orb}(f,(\gamma, v,  v^*\varepsilon))\omega(\gamma,v, v^*\varepsilon)$$
for all $ \varepsilon$ and $\varepsilon_1$ such that $\textup{Nm}_{E[\delta]/F[\delta]}(\varepsilon_1)=\varepsilon$. Recall that the transfer factor $\omega$ is independent of the the $V^*$ part.

Choose $\varepsilon_1$ and $\varepsilon$ that are very close to 0. Apply the germ expansions in section \ref{general linear germ} and section \ref{unitary germ}. The lemma below implies we may subtract the leading terms correspond to $\Lambda_2=\emptyset$. Therefore

\begin{equation}\label{C1}
\int_{\textup{U}(W)/T_\delta}f_W(g\delta g^{-1},0)d\overline{g}=\sum_{\Lambda_1 \subseteq S_1} \textup{Orb}(f,(\gamma,v_{\Lambda_1},v^*_{S_1 \backslash \Lambda_1}))\omega(\gamma,v, v^*). 
\end{equation}

Notice that we do not have the factor $\prod_{S_1 \backslash \Lambda_1}\chi(\varepsilon_i)$ on the right side since $\varepsilon$ is a norm.

 \begin{lemma} 
 Suppose there exists constant $c_\Lambda$ independent of $\varepsilon$ such that for all the $\varepsilon$ close enough to $0$, $$\sum_{\Lambda \subseteq S}c_\Lambda \log_q|\varepsilon |_\Lambda  = 0$$ then all the $c_\Lambda$ are $0$.
 \end{lemma}
 \begin{proof}
 This easily follows by induction on $|S|$ and changing one of the $\varepsilon_i$ by its square.
 \end{proof}

\subsection{Inverse the identity} \label{inverse}

Theorem \ref{Nilpotent Identity} is obtained from (\ref{C1}) by considering the stable conjugacy class of $\delta$, each conjugacy class will give an identity and inversing these identities will give theorem \ref{Nilpotent Identity}.

Fix a regular semisimple $(\gamma, v, v^*) \in \mathfrak{gl}(V) \times V \times V^*$ and assume $\gamma$ is also regular semisimple. For any $x \in F[\gamma]^\times$, we denote by $(\delta_x,w_x)$ the orbit that matches $(\gamma, v,v^*x)$. Consider the map $x \mapsto \delta_x$.

\begin{lemma} \label{get delta_x}
We use the notation S as in proposition $\ref{bijection}$. The map $x \mapsto \delta_x$ defines an isomorphism of $H^1(F,T_x)$ torsor $$ F[\gamma]^\times/\textup{Nm}_{E[\gamma]/F[\gamma]}(E[\gamma]^\times) \rightarrow S.$$
\end{lemma}

\begin{proof}
The image lies in a single $S$ since they are all stably conjugate to $\gamma$. The map factors through the quotient since if $(\gamma,v,v^*)$ and $(\delta,w)$ match, then $(\gamma,v, v^*\varepsilon)$ and $(\delta,\varepsilon_1 w)$ also match with Nm$(\varepsilon_1)=\varepsilon$ (lemma \ref{change match}).

Let us show the map is a bijection. Let $F[\gamma]=\prod_{i=1}^m F_i$, $V_i=F_iV$, and $v=(v_i)$ in the decomposition $V=\bigoplus V_i$. Define $\Delta_i(\gamma,v,v^*)$ to be $\chi($det$(\{\langle v^*_i,\gamma^{a+b}v_i\rangle\}_{a,b}))$. Similarly define $\Delta_i(\delta,w)$. The invariant $\Delta_i$ are the same for matching orbits. 

Notice that ($\Delta_1(\gamma, v, v^*x),...,\Delta_m(\gamma, v, v^*x)$) $\in \{\pm 1\}^m$ is uniquely determined by $x$, and ($\Delta_1(\delta, w),...,\Delta_m(\delta, w)$) $\in \{\pm 1\}^m$ uniquely determines the isomorphism class of the Hermitian space $E_iW $, which uniquely determines  elements in $S$. Therefore, the map is a bijection.

Since $H^1(F,T_x)$ acts on $S$ by acting on the isomorphism class of $E_iW$, it is obvious that the map commutes with the $H^1(F,T_x)$ action.

\end{proof}

Apply (\ref{C1}) to the orbits $(\gamma, v, v^*x)$ and $(\delta_x,w_x)$:

$$\int_{\textup{U}(W)/T_{\delta_x}}f_W(g\delta_x g^{-1},0)d\overline{g}=\sum_{\Lambda_1 \subseteq S_1} \textup{Orb}(f,(\gamma,v_{\Lambda_1},(v^*x)_{S_1 \backslash \Lambda_1}))\omega(\gamma,v, v^*).$$

After a change of variable in $v^*$ and notice that the transfer factor $\omega$ is independent of the $V^*$ factor, this becomes

$$\int_{\textup{U}(W)/T_{\delta_x}}f_W(g\delta_x g^{-1},0)d\overline{g}=\sum_{\Lambda_1 \subseteq S_1} \prod_{S_1 \backslash \Lambda_1}\chi(x) \textup{Orb}(f,(\gamma,v_{\Lambda_1},v^*_{S_1 \backslash \Lambda_1}))\omega(\gamma,v, v^*).$$

There is a natural pairing between $\prod_{S_1} (F_i^\times/$Nm$(E_i^\times))$ and $\prod_{S_1} \mathbb{Z}/2\mathbb{Z}$ with value in $\pm 1$, namely the direct sum of the non-trivial pairing between $\mathbb{Z}/2\mathbb{Z}$ and $\mathbb{Z}/2\mathbb{Z}$. We write this pairing as $\langle \ , \ \rangle$ (This conflicts with our previous notations where we use it as Hermitian forms).  We identify $\Lambda \subseteq S_1$ with an element in $\prod_{S_1} \mathbb{Z}/2\mathbb{Z}$ in the following way: the corresponding element is $1$ in its $i$-th component if and only if $i \notin \Lambda$. By abuse of notation, we also write this element in $\prod_{S_1} \mathbb{Z}/2\mathbb{Z}$  as $\Lambda$.  

The pairing enables us to rewrite the above identity as

$$\int_{\textup{U}(W)/T_{\delta_x}}f_W(g\delta_x g^{-1},0)d\overline{g}=\sum_{\Lambda_1 \subseteq S_1} \langle \Lambda_1, x\rangle \textup{Orb}(f,(\gamma,v_{\Lambda_1},v^*_{S_1 \backslash \Lambda_1}))\omega(\gamma,v, v^*).$$

Finally, apply Fourier transform on the Abelian group $\prod_{S_1} \mathbb{Z}/2\mathbb{Z}$. This inverses the identity and it becomes 
\begin{theorem} \label{Nilpotent Identity}
Using the notation in this subsection, we have
\begin{equation} \label{identity}
\omega(\gamma,v, v^*) \textup{Orb}(f,(\gamma,v_{\Lambda},v^*_{S_1 \backslash \Lambda}))= \sum_x \langle \Lambda,x \rangle \int_{\textup{U}(W_x)/T_{\delta_x}}f_W(g\delta_x g^{-1},0)d\overline{g}.
\end{equation}
\end{theorem}
Let us reinterpret this formula. We start by choosing regular semisimple $\gamma$. This defines the conjugacy class $S$ on the unitary side. Then we choose some $\Lambda$. This defines a character of $H^1(F,T)$. Finally, we choose auxiliary $v$ and $v^*$. This gives the quantity on the left side of the formula and also determines a base point $\delta$ of S (($\delta,w$) matches ($\gamma, v, v^*$)). It is a funny exercise to show the formula we get are equivalent for different choice of $v$ and $v^*$. So it only depends on the choice of $\gamma$ and $\Lambda$. Also, by choosing appropriate $\Lambda$, we get all the possible $\kappa$-orbit integrals.

\section{Parabolic descent}

We shall use a computation similar to Harish-Chandran's parabolic descent of orbit integrals. This will be used in the next section to prove the main theorem. 

Let $f \in C_c^\infty(\mathfrak{gl}(V)\times V\times V^*)$ and $V=V_1 \bigoplus V_2$. We define the parabolic descent of $f$ as a function $f^P  \in C_c^\infty(\prod_1^2 \mathfrak{gl}(V_i)\times V_i \times V^*_i)$.

$$f^P(\lambda,v,v^*)=\int_{\mathfrak{n}}f_K\bigg(\lambda+
\begin{pmatrix}
    0 & n  \\
    0 & 0 
  \end{pmatrix},v,v^*\bigg)dn$$
where $\mathfrak{n}=$Hom$(V_2,V_1)$ and 
 
 $$f_K = \int_{\textup{GL}_n(O_F)}f(kxk^{-1},kv,v^*k^{-1})\chi(k)dk.$$

Let $\lambda=(\lambda_1,\lambda_2) \in \mathfrak{gl}(V_1) \times \mathfrak{gl}(V_2)$ be regular semisimple in $\mathfrak{gl}(V)$.

\begin{lemma} \label{lemma descent}
When $v_2=v_1^*=0$ and $(\lambda, v, v^*)$ satisfies the condition in equation \textup{(\ref{converge})} to make sense its ``nilpotent orbit integral'', we have

$$\textup{Orb}(f,(\lambda,v,v^*))=\textup{Orb}(f^P,(\lambda,v,v^*))|D(\lambda)|_F^{-1}$$
with  $$D(\lambda) = \prod_{x_1,x_2 }(x_1-x_2)$$
and $x_1 (\textup{resp. } x_2)$ runs through all the eigenvalues of $\lambda_1(\textup{resp. } \lambda_2)$ over $\overline{F}$.

\end{lemma}
\begin{proof}
Use $G=KNM$ to compute the integral with the usual meaning of $K, M, N$ as subgroups of GL$(V)$. Let $T$ be the centralizer of $\lambda$ in GL$(V)$ and $T_1$ the components of $T$ belongs to $S_1$. We omit the discussions of the choice of Haar measures on various integrations.
\begin{align*}
\textup{Orb}&(f,(\lambda,v,v^*)) \\
  &=\int_{G/T}\int_{T_1}f(g\lambda g^{-1},gtv,v^*(gt)^{-1})|t|^{\pm s}\chi(t)\chi(g)\bigg \rvert_{s=0} \\
  &=\int_{K}\int_{N}\int_{M/T}\int_{T_1}f(knm\lambda (knm)^{-1},knmtv,v^*(knmt)^{-1})|t|^{\pm s}\chi(t)\chi(knm)\bigg \rvert_{s=0} \\
  &=\int_{N}\int_{M/T}\int_{T_1}f_K(nm\lambda (nm)^{-1},nmtv,v^*(nmt)^{-1})|t|^{\pm s}\chi(t)\chi(m)\bigg \rvert_{s=0} \\
  &=\int_{N}\int_{M/T}\int_{T_1}f_K(nm\lambda (nm)^{-1},mtv,v^*(mt)^{-1})|t|^{\pm s}\chi(t)\chi(m)\bigg \rvert_{s=0}.
\end{align*}
The last identity follows from the assumption $v_2=v_1^*=0$.

Now this is almost $\textup{Orb}(f^P,(\lambda,v,v^*))$ except that the upper right block of $nm\lambda (nm)^{-1}$ equals $nm_2\lambda_2m_2^{-1}-m_1\lambda_1m_1^{-1}n$. Here $m=(m_1,m_2)$ and by an abuse of notation, we also use $n$ to denote the upper right block of the corresponding unipotent matrix.

Consider the change of variables: 

$$nm_2\lambda_2m_2^{-1}-m_1\lambda_1m_1^{-1}n=n^{\prime}.$$ 

The lemma follows since $D(\lambda)^{-1}$ is the jacobian for this change of variables.
\end{proof}
\begin{lemma}
The map $f \mapsto f^P$  is preserved by the Fourier transforms on the $\mathfrak{gl}$(V) factor. \end{lemma}
\begin{proof}
As before, the Fourier transforms are defined by the non-degenerate symmetric bilinear form tr$(XY)$ and a fix additive character $\psi: F \rightarrow \mathbb{C}^\times$. We can assume $f_K=f$ since $\mathcal{F}$ commutes with averaging over $K$. Since $f^P$ only change the $\mathfrak{gl}(V)$ factor, we can ignore the $V$ and $V^*$ factors and therefore assume $f \in C_c^\infty(\mathfrak{gl}(V))$. We have

\begin{equation}\label{B1}
(\mathcal{F}f)^P(\lambda)=\int_{\mathfrak{n}} \mathcal{F}f\bigg(\lambda+\begin{pmatrix}
    0 & n  \\
    0 & 0 
  \end{pmatrix}\bigg)=\int_{\mathfrak{n}} \int_{\mathfrak{gl}(V)}f(X)\psi(\textup{tr}(X_1\lambda_1+X_3n+X_4\lambda_4))
\end{equation}  
   with 

 \[
 X=
  \begin{pmatrix}
    X_1 & X_2  \\
    X_3 & X_4 
  \end{pmatrix}.
\]

Consider the term tr$(X_3n)$. The integration over $\mathfrak{n}$ is the same as an integration over $S_m=\pi_F^{-m}\mathfrak{n}(O_F)$ for $m$ sufficiently large. And $\int_{S_m} \psi(\textup{tr}(X_3n))$ is 0 unless $X_3$ is very close to $0$. Hence for any small $\delta$, we may find $m$ such that the domain of the integral in (\ref{B1}) can be replaced by $\{(n,X)|n \in S_m, |X_3| \leqslant \delta\}$. Take $\delta$ small enough, we may replace $f(X)$ in (\ref{B1}) by $f\bigg(\begin{pmatrix}
    X_1 & X_2  \\
    0 & X_4 
  \end{pmatrix}\bigg).$

 Now, after integrating over $\mathfrak{n}$ and $X_3$, this equals $\mathcal{F}(f^P)$.

\end{proof}

\section{Proof of the main theorem} \label{final}

In this section, we prove our main theorem. Let us fix notations that will be used in the proof. As before, $E/F$ is a quadratic extension of non-archimedean local fields of characteristic zero. Let $\chi$ be the quadratic character of $F^\times$ corresponding to $E/F$. We will write $W_{n,0}$ and $W_{n,1}$ as the two isomorphism classes of non-degenerate Hermitian spaces of dimension $n$ for $E/F$. We use $W_{n,0}$ to denote the split one. Also, let $V_n=F^n$ be the vector space over $F$ of dimension $n$.

\begin{theorem}[Main Theorem] \label{Main Theorem}
Let $W$ be a Hermitian space of dimension $n$ for $E/F$. Let $n=a+b$. Define $G=\textup{U}(W)$ and it has endoscopic group $H=\textup{U}(W_{a,0})\times \textup{U}(W_{b,0})$. 

\textup{(1)} For each $f \in C_c^\infty(\mathfrak{g})$, there exists a transfer $f^H \in C_c^\infty(\mathfrak{h})$.

\textup{(2)} If $f$ and $f^H$ are matching functions, then for the Fourier transfrom $\mathcal{F}$ and $\mathcal{F}^H$ \textup{(}defined by \textup{tr(}$XY$\textup{))}, $\mathcal{F}(f)$ and $\mathcal{F}^H(f^H)$ are also matching functions.

\end{theorem}

\begin{proof}

We first prove the existence of transfer. The transfer is defined in several steps. Start with $f \in C_c^\infty(\mathfrak{u}(W))$,
\begin{itemize}
\item Choose any $F \in C_c^\infty(\mathfrak{u}(W)\times W)$ such that $F(x,0)=f(x)$.
\item Define $\phi \in C_c^\infty(\mathfrak{gl}(V_n)\times V_n \times V_n^*)$ as the Jacquet-Rallis transfer of $\{F,0\}$.
\item Define $\phi^{a,b} \in C_c^\infty(\mathfrak{gl}(V_a) \times \mathfrak{gl}(V_b) \times V_n \times V_n^*)$ as the parabolic descent of $\phi$.
\item Using the product of Jacquet-Rallis transfers in dimension $a$ and $b$, we can transfer $\phi^{a,b}$ to four functions $\{F^{a,b}_{i,j}\}({i, j = 0,1})$, where $$F^{a,b}_{i,j} \in C_c^\infty(\mathfrak{u}(W_{a,i})\times W_{a,i} \times \mathfrak{u}(W_{b,j})\times W_{b,j}).$$
\item Define $f^{a,b}_{i,j} \in \mathfrak{u}(W_{a,i}) \times \mathfrak{u}(W_{b,j})$ as $f^{a,b}_{i,j}(x,y)=F^{a,b}_{i,j}(x,0,y,0).$
\item Let $\widetilde{f^{a,b}_{i,j}}$ be the Jacquet-Langlands transfer of $f^{a,b}_{i,j}$ to $\mathfrak{u}(W_{a,0}) \times \mathfrak{u}(W_{b,0}).$
\item Finally, define $f^{a,b}=\widetilde{f^{a,b}_{0,0}}-\widetilde{f^{a,b}_{0,1}}+\widetilde{f^{a,b}_{1,0}}-\widetilde{f^{a,b}_{1,1}}.$
\end{itemize}

We claim $f^{a,b}$ is an endoscopic transfer of $f$. Choose elements $\delta \in \mathfrak{u}(W)$ and $(\delta_1, \delta_2) \in \mathfrak{u}(W_{a,0}) \times \mathfrak{u}(W_{b,0})$ whose orbits are matching. We first verify (\ref{endos}) for these orbits. Let $W^\prime = W_{a,0} \bigoplus W_{b,0}$, recall in definition \ref{nice}, $(\delta_1, \delta_2)$ gives a nice matching element $\delta^\prime \in \mathfrak{u}(W^\prime)$.

Choose auxiliary $w_1 \in W_{a,0}$ and $w_2 \in W_{b,0}$ such that $(\delta_i, w_i)$ is regular semisimple. Let $(\gamma_i,v_i,v_i^*)$ be a Jacquet-Rallis transfer of $(\delta_i, w_i)$. The element $((\gamma_1,\gamma_2), (v_1,v_2), (v_1^*,v_2^*))$ is naturally a regular semisimple element in $\mathfrak{gl}(V_n)\times V_n \times V_n^*$ that is a Jacquet-Rallis transfer of $(\delta^\prime,(w_1, w_2))$. 

The equation (\ref{endos}) follows from a sequence of identities. To simplify the equation, for the moment, we use $C_k$ to represent all the transfer factors appearing in these identities. Recall that $F[\delta]=\prod F_k$, $S_1=\{k|F_k \nsupseteq E\}$ and $S_1(\delta_i)=\{k \in S_1|F_k\textup{ comes from }\delta_i\}$.

\begin{align*}
\textup{Orb}^{\textup{St}}(f^{a,b},(\delta_1, \delta_2)) 
  & =\textup{Orb}(G^{a,b},((\gamma_1,\gamma_2), ((v_1)_{S_1(\delta_1)},0), (0,(v_2^*)_{S_1(\delta_2)}))) \times C_1 \\
 & =\textup{Orb}(G,((\gamma_1,\gamma_2), ((v_1)_{S_1(\delta_1)},0), (0,(v_2^*)_{S_1(\delta_2)}))) \times C_2\\
 &= \textup{Orb}^{\kappa}(f,\delta) \times \kappa(\textup{inv}(\delta^\prime,\delta)) \times C_3.
\end{align*}

The first identity follows by taking the product of (\ref{identity}) applied to the case when dim$V=a$ and $\Lambda=S_1(\delta_1)$ and the case when dim$V=b$ and $\Lambda=\emptyset$. The second identity follows from parabolic descent (lemma \ref{lemma descent}). The third identity follows by applying (\ref{identity}) to the case when dim$V=n$ and $\Lambda=S_1(\delta_1)$. 

We remark that $\kappa$-orbit integral here is different from the classical one since it contains conjugacy classes in both Hermitian spaces. We have shown that these conjugacy classes is a natural torus under $H^1(F,T_\delta)$ and $\kappa$ is character of $H^1(F,T_\delta)$. Therefore we can extend the usual $\kappa$ orbit integral to 

Let us explain that the third identity does give the $\kappa$-orbit integral defined using the endoscopic group $H$. This means we will need to verify the identity
\begin{equation} \label{D2}
\kappa(\textup{inv}(\delta^\prime, \delta_x))= \langle S_1(\delta_1), x \rangle.
\end{equation}

Here $(\delta_x, w_x)$ is the Jacquet-Rallis transfer of $((\gamma_1,\gamma_2), (v_1,v_2), (v_1^*,v_2^*)x)$ with $x \in F[\gamma]^\times$. By proposition \ref{get delta_x}, the canonical isomorphism between $F[\gamma]^\times/\textup{Nm}_{E[\gamma]/F[\gamma]}(E[\gamma]^\times)$ and $H^1(F,T_\gamma)$ maps $x$ to $\textup{inv}(\delta, \delta_x)$ (proposition \ref{get delta_x}) and maps $\kappa$ to $\langle S_1(\delta_1), \ \rangle$ (proposition \ref{definition of kappa}). This implies (\ref{D2}).

Finally, we compute the transfer factor. In the above computation, we have introduced three transfer factors: two come from the Jacquet-Rallis transfer and one comes from the parabolic descent. We need to prove their product gives the transfer factor $\chi(D(\delta))|D(\delta)|_F$. The transfer factor in the parabolic descent equals $|D(\delta)|_F$. So it remains to show the product of the two transfer factors in the Jacquet-Rallis transfer equals  $\chi(D(\delta))$.

To define the transfer factor in the Jacquet-Rallis case, we fix $\Omega_1 \in \wedge^{\textup{top}}V_a$, $\Omega_2 \in \wedge^{\textup{top}}V_b$, and take $\Omega \in \wedge^{\textup{top}}V$ defined by $\Omega_1 \wedge \Omega_2$. For careful readers, we should fix this transfer factor before defining the Jacquet-Rallis transfer of functions.  Recall we have defined $\omega(\gamma, v, v^*)$ in  (\ref{transfer}). The desired equality for transfer factors reduces to

$$\omega(\gamma, v, v^*)=\chi(D(\gamma))\omega(\gamma_1, v_1, v_1^*)\omega(\gamma_2, v_2, v_2^*),$$
where $\gamma=(\gamma_1, \gamma_2), v=(v_1, v_2)$ and $v^*=(v^*_1, v^*_2)$. To prove this, diagonalize $\gamma_1$ and $\gamma_2$ in $\overline{F}$ and compute both sides.

To verify $f^{a,b}$ is an endoscopic transfer of $f$, we still need to verify the requirement (\ref{more}). Namely, consider $(\delta_1,\delta_2)$ that is not a transfer from $\mathfrak{u}(W)$. This happens only when $S_1(\delta)=\emptyset$. In this case, $W^\prime=W_{n,0}$ and $W^\prime \ncong W$. The stable conjugacy class of the nice matching element $\delta^\prime$ of $(\delta_1,\delta_2)$ is the unique conjugacy class contained in $\mathfrak{u}(W_{n,0})$. Following the same computation as before, we find $$\textup{Orb}^{\textup{St}}(f^{a,b},(\delta_1,\delta_2))=0.$$

This finishes the proof of the existence part. To prove the Fourier transform part, notice that our construction commutes with Fourier transforms. For this commutativity, it is important that the $(-1)^{n-1}$ in theorem \ref{JR commute} and proposition \ref{E1} exactly cancels. We have shown that for any $f$, there exists $f^H$ such that $f$ matches $f^H$ and $\mathcal{F}(f)$ matches  $\mathcal{F}(f^H)$. A theorem of Waldspurger (Proposition A in \cite{Wa3}) ensures that for a fixed $f$, the validity of commutativity does not depend on the choice of $f^H$. This finishes our proof. 
\end{proof}

The same idea also applies to the fundamental lemma.

\begin{theorem}
For any odd prime $p$ and consider the fundamental lemmas on the Lie algebras. Then the fundamental lemma in the Jacquet-Rallis case $($conjecture \textup{\ref{Fundamental}}$)$ implies the fundamental lemma in the endoscopic case. 
\end{theorem}

\begin{proof}
We leave it to the reader to give the formulation of the fundamental lemma in the unitary Lie algebras. To prove the above theorem, we start with the right function on $\mathfrak{u}(W_{n,0})$. Then the Jacquet-Rallis fundamental lemma and a simple calculation of parabolic descent implies that we get the right function on  $\mathfrak{u}(W_{a,0})\times  \mathfrak{u}(W_{b,0})$ following the process as in the proof of theorem \ref{Main Theorem}.
\end{proof}

\begin{remark}
In \cite{Kazhdan}, Kazhdan and Varshavsky have shown that the fundamental lemma in the endoscopic case is implied by theorem $\ref{LS commute}$ $($at least when $p$ is large$)$. In the unitary case, theorem $\ref{Main Theorem}$ give a direct proof of theorem $\ref{LS commute}$. Hence, we also give a genuinely different proof of the endoscopic fundamental lemma. 
\end{remark}

\end{document}